\documentclass[11pt,twoside]{article}
\usepackage{amsmath,amsbsy,amsfonts}
\usepackage{graphicx,color}
\everymath{\displaystyle}
\newtheorem{theorem}{Theorem}[section]
\newtheorem{lemma}[theorem]{Lemma}
\newtheorem{proposition}[theorem]{Proposition}
\newtheorem{corollary}[theorem]{Corollary}

\newenvironment{proof}[1][Proof]{\noindent\textbf{#1.} }{\ \rule{0.5em}{0.5em}}

\begin{document}

\title{\bf Existence of multi-bump solutions for a class of elliptic problems involving the biharmonic operator}

\author{Claudianor O. Alves\thanks{C. O. Alves was partially supported by CNPq/Brazil
		301807/2013-2 and INCT-MAT, coalves@dme.ufcg.edu.br}\, , \, \  Al\^{a}nnio B. N\'{o}brega \thanks{alannio@dme.ufcg.edu.br}\vspace{2mm}
	\and {\small  Universidade Federal de Campina Grande} \\ {\small Unidade Acad\^emica de Matem\'{a}tica} \\ {\small CEP: 58429-900, Campina Grande - Pb, Brazil}\\}

\date{}
\maketitle

\begin{abstract}
	Using variational methods, we establish existence of multi-bump solutions for the
	following class of problems 
	$$
	\left\{
	\begin{array}{l}
	\Delta^2 u  +(\lambda V(x)+1)u  =  f(u), \quad \mbox{in} \quad \mathbb{R}^{N},\\
	u \in H^{2}(\mathbb{R}^{N}),
	\end{array}
	\right.
	$$
	where $N \geq 1$,  $\Delta^2$ is the biharmonic operator, $f$ is a continuous
	function with subcritical growth and $V : \mathbb{R}^N \rightarrow \mathbb{R}$ is a continuous function
	verifying some conditions.

	\vspace{0.3cm}
	
	
	\vspace{0.3cm}
	
\end{abstract}

\section{Introduction}

In this paper,  we are concerned with the existence of multi-bump solutions for the following class of problems
\begin{equation}\label{1}
\left\{ \begin{array}{cc}
\Delta^2 u  +(\lambda V(x)+1)u  & =  f(u), \quad \mbox{in} \quad \mathbb{R}^{N}, \\
u \in H^2(\mathbb{R}^N);\,\,\,\,\,\,\,\,\,\,\,\,\,\,\,\,&  
\end{array}
\right.
\end{equation}
where $N \geq 1$, $\Delta^{2}$ denotes the biharmonic operator, $\lambda >0$ is a positive parameter and $f:\mathbb{R} \rightarrow \mathbb{R}$ is a $C^1$ function verifying  the following hypotheses:
\begin{description}
	\item[$(f_1)$] $f(0)=f'(0)=0.$
	\item[$(f_2)$] $\liminf_{t \rightarrow +\infty} \frac{\left|f'(t)\right|}{\left|t\right|^{q-2}}<+\infty,
	$ for $q \in (2,2_*)$ where $$2_*=
	\left\{ \begin{array}{c}
	\frac{2N}{N-4},\quad N\geq 5 \\
	+\infty, \quad 1\leq N \leq 4  .
	\end{array}
	\right.$$
	\item[$(f_3)$] There is $\theta >2$ such that 
	$$
	0<\theta F(t) \leq f(t) t,\quad \mbox{for} \quad t\neq 0.
	$$
	\item[$(f_4)$] $\frac{f(t)}{\left|t\right|}$ is an increasing function for $t\neq 0$.
\end{description}

Related to the potential $V:\mathbb{R}^N \rightarrow
\mathbb{R}$, we assume the following assumptions :
\begin{description}
	\item[$(V_1)$]  $V(x)\geq 0, \ \forall \ x\in \mathbb{R}^N$;
	\item[$(V_2)$]  $\Omega= int V^{-1}(\{0\})$ is a non-empty bounded open set with smooth boundary $\partial \Omega$. Moreover, $\Omega$ has $k$ connected components, more precisely,
	\begin{itemize}
		\item $\Omega=\bigcup_{j=1}^{k}\Omega_j;$
		\item $dist(\Omega_i,\Omega_j)>0,\ i\neq j.$
	\end{itemize} 
	\item[$(V_3)$]  There is $M_0>0$ such that $ |\{ x\in \mathbb{R}^N ;\, V(x)\leq
	M_0\}|<+\infty.$
\end{description}
Hereafter, if $A \subset \mathbb{R}^{N}$ is a mensurable set, $|A|$ denotes its Lebesgue's measure.

In the last years, problems involving the biharmonic operator have been studied by many researchers, in part because this operator helps to describe the mechanical vibrations of an elastic plate, which among other things describes the traveling waves in a suspension bridge, see \cite{BDF,FG,G,GK,LM}. On the other hand, the biharmonic operator draws attention by the difficulties encountered when trying to adapt known results for the Laplacian, for example, we cannot always rely on a maximum principle, and also, if $u$ belongs $ H ^ 2 (A)$,  we cannot claim that $u ^ \pm$ belong to $ H ^ 2 (A)$.  Recently, many authors have studied various problems with the biharmonic operator, see for example, \cite{BG,JQ,P1,P2,YT,ZWZ}. However, related to the existence
of multi-bump solutions for an equation as (\ref{1}), as far as we
know, there is no results in this direction.

In \cite{D&T}, Ding and Tanaka have considered the problem
\begin{equation}\label{2}
\left\{ \begin{array}{l}
-\Delta u  +(\lambda V(x)+Z(x))u  = u^p, \quad \mbox{in} \quad \mathbb{R}^{N}, \\
u>0,\ \mbox{in}\ \mathbb{R}^N,
\end{array}
\right.
\end{equation}
with $p \in \left(1,\frac{N+2}{N-2}\right)$  and $ N \geq 3$. In that paper, it was showed that the problem (\ref{2}) has at least $2^k-1$ solutions for $\lambda$ large enough, which are called multi-bump solutions. These solutions have the following characteristics : \\

\noindent For each non-empty subset $\Gamma \subset \{1,2,\cdots,k\}$  and $\varepsilon>0$ fixed, there is a $\lambda^*>0$ such that, (\ref{2}) possesses a solution $u_{\lambda}$, for $\lambda \geq \lambda^*=\lambda^*(\varepsilon)$, satisfying:
$$
\left| \int_{\Omega_j}\left[\left|\nabla u_{\lambda}\right|^2+(\lambda V(x)+Z(x))u_{\lambda}^2\right]-\left( \frac{1}{2}-\frac{1}{p+1}\right)^{-1}c_j\right|< \varepsilon,\ \forall j \in \Gamma
$$
and
$$\int_{\mathbb{R}^N\setminus \Omega_{\Gamma}}\left[ \left| \nabla u_{\lambda}\right|^2+ u_{\lambda}^2\right]dx< \varepsilon,$$
where $\Omega_{\Gamma}= \bigcup_{j \in \Gamma}\Omega_j$ and $c_j$ is the
minimax level of the energy functional related to the problem
\begin{equation}\label{3}
\left\{ \begin{array}{c}
- \Delta u  +Z(x)u  = u^p, \,\, \mbox{in} \,\, \Omega_j, \\
u>0, \,\, \mbox{in} \,\, \Omega_j, \\
u=0,\ \mbox{on}\ \partial\Omega_j  .
\end{array}
\right.
\end{equation}

We also highlight the papers due to  Alves, de Morais Filho and Souto in \cite{A-M-S}, Alves and Souto in \cite{AS}, where the authors have considered a problem of type (\ref{2}), assuming that $f$ has a critical growth for the case $N \ge 3$ and exponential
critical growth when $N = 2$, respectively.   We emphasize that in the above mentioned papers, the assumption $(V_3)$ was not assumed. 

In all the above mentioned papers, it was essential the method developed in \cite{Del&Felm},  which consists in modifying the nonlinearity to obtain a new problem, whose energy functional  associated satisfies the $(PS)$ condition. After that, making some estimates, it is possible to prove that the solutions obtained for the modified problem are also solutions for the original problem when $\lambda$ is large enough. However, in our opinion, it is not clear that the method developed in \cite{Del&Felm} can be used for our problem, because we are working with biharmonic operator.  To overcome this difficulty, we have developed a new approach to get multi-bump avoiding the penalization on  the nonlinearity. Our inspiration comes from an approach  used in Bartsch \& Wang \cite{B&W1, B&W2}. Here, we modify  the sets where we will apply the Deformation Lemma, see Sections 4 and 5 for more details. 

\vspace{0.5 cm}

Our main result is the following

\begin{theorem}\label{T1}
	Suppose that $(f_1)-(f_4)$ and $(V_1)-(V_3)$ hold. Then, for each non-empty subset
	$\Gamma \subset \{1,\cdots,k\}$and $\varepsilon>0$ fixed, there is a $\lambda^*=\lambda^*(\varepsilon)>0$ such that, (\ref{1}) possesses a solution $u_{\lambda}$, for $\lambda \geq \lambda^*$, satisfying: 
	$$\left| \frac{1}{2}\int_{\mathbb{R}^N}\left[\left|\Delta u_{\lambda}\right|^2+(\lambda V(x)+1)\left|u_{\lambda}\right|^2\right]dx-\int_{\mathbb{R}^N}F(u_{\lambda})dx-c_j\right|< \varepsilon, \forall j \in \Gamma$$
	and
	$$\int_{\mathbb{R}^N\setminus \Omega_{\Gamma}}\left[\left|\Delta u_{\lambda}\right|^2+\left|u_{\lambda}\right|^2\right]dx < \varepsilon,$$
	where $\Omega_{\Gamma}=\cup_{j\in \Gamma}\Omega_j$ and $c_j$ is the
	minimax level of the energy functional related to the problem:
	\begin{equation}\label{4}
	\left\{ \begin{array}{c}
	\Delta^2 u  +u  = f(u), \quad \mbox{in} \quad \Omega_j \\
	u=\frac{\partial u}{\partial \eta}=0,\quad  \mbox{on} \quad \partial\Omega_j  .
	\end{array}
	\right.
	\end{equation}
\end{theorem}

\section{The $(PS)_c$ Condition}

In this section, we fix some notations and show some properties of the energy functional associated with (\ref{1}), for example, we will show that for each $c\geq 0$, the functional $I_{\lambda}$ satisfies the $(PS)_c$ condition, since that $\lambda$ is suitably chosen. 

To begin with, we recall that the energy functional $I_{\lambda}:E_{\lambda}\rightarrow \mathbb{R}$ associated with the problem  $(\ref{1})$ is given by 
$$
I_{\lambda}(u)=\frac{1}{2}\int_{\mathbb{R}^N}\left[\left|\Delta u\right|^2+(\lambda V(x)+1)\left|u\right|^2\right]dx-\int_{\mathbb{R}^N}F(u)dx,
$$
where
$$
E_{\lambda}=\left\{ u\in H^2(\mathbb{R}^N);\ \int_{\mathbb{R}^N}V(x)\left|u\right|^2 dx <+\infty\right\}.
$$
The subspace $E_{\lambda}$ endowed with the inner product
$$
(u,v)_{\lambda}= \int_{\mathbb{R}^N} \left[\Delta u\Delta v+(\lambda V(x)+1)uv\right]dx,
$$
is a Hilbert space and the norm generated by this inner product will be denoted by $\|\cdot\|_{\lambda}$. 

Hereafter, if $\Theta \subset \mathbb{R}^{N}$ is a mensurable set, we denote by $E_\lambda(\Theta)$ the space $H^{2}(\Theta)$ endowed with the the inner product
$$
(u,v)_{\lambda,\Theta}= \int_{\Theta} \left[\Delta u\Delta v+(\lambda V(x)+1)uv\right]dx.
$$
The norm associated with this inner product will be denoted by $\|\cdot\|_{\lambda,\Theta}$.

Next, we will show some \underline{}technical lemmas, whose proofs follow with the same type of arguments found in \cite{B&W1, B&W2}. However for the readers' convenience we will write their proofs. 

\begin{lemma}\label{l1}
	Let $\{u_n\} \subset E_{\lambda}$ be a $(PS)_c$ sequence  for
	$I_{\lambda},$ then $\{u_n\}$ is bounded in $E_{\lambda}$. Furthermore, $c\geq 0.$
\end{lemma}
\begin{proof}
	Since $\{u_n\}$ is a $(PS)_c$ sequence, we have that 
	$$
	I_{\lambda}(u_n)\rightarrow c\ \mbox{and}\ I'_{\lambda}(u_n) \rightarrow 0.
	$$
	Thereby, for $n$ large enough,
	\begin{equation}\label{5}
	I_{\lambda}(u_n)-\frac{1}{\theta}I'_{\lambda}(u_n)u_n \leq c+1+\left|\left|u_n\right|\right|_{\lambda}.
	\end{equation}
	On the other hand,
	$$
	I_{\lambda}(u_n)-\frac{1}{\theta}I'_{\lambda}(u_n)u_n=\left( \frac{1}{2}-\frac{1}{\theta}\right)\|u_n\|^2_{\lambda}+\int_{\mathbb{R}^N}\left[ \frac{1}{\theta}f(u_n)u_n-F(u_n)\right]\,dx.
	$$
	Then, by $(f_3)$, 
	\begin{equation}\label{6}I_{\lambda}(u_n)-\frac{1}{\theta}I'_{\lambda}(u_n)u_n\geq \left( \frac{1}{2}-\frac{1}{\theta}\right)\|u_n\|^2_{\lambda}.
	\end{equation}
	Gathering $ (\ref{5}) $ and $ (\ref{6}) $, we get 
	$$
	\left( \frac{1}{2}-\frac{1}{\theta}\right)\|u_n\|^2 _{\lambda} \leq c+1+\|u_n\|_{\lambda},
	$$
	showing that $\{u_n\}$ is bounded. Using the boundedness of $\{u_n\}$ and (\ref{5}), we see that
	\begin{equation}\label{7}0\leq\left(
	\frac{1}{2}-\frac{1}{\theta}\right)\|u_n\|^2 _{\lambda} \leq
	c+o_n(1).
	\end{equation}
	Taking the limit $n \rightarrow +\infty$, it follows that $c \geq 	0.$
\end{proof}
\begin{corollary}\label{c1}
	Let $\{u_n\} \subset E_{\lambda}$ be a $(PS)_0$ sequence for
	$I_{\lambda}.$ Then, $u_n\rightarrow 0$ in $E_\lambda$.
\end{corollary}
\begin{proof}
This corollary is an immediate consequence of the arguments used in the proof of Lemma $\ref{l1}$.
\end{proof}

\begin{lemma}\label{l2}
	Let $\{ u_n\}$ be a $(PS)_c$ sequence for
	$I_{\lambda}$ with $c \geq 0$. If $u_n \rightharpoonup u$ in $E_{\lambda},$ then
	\begin{eqnarray}
	I_{\lambda}(v_n)-I_{\lambda}(u_n)+I_{\lambda}(u) &=& o_n(1) \nonumber\\
	I'_{\lambda}(v_n)-I'_{\lambda}(u_n)+I' _{\lambda}(u) &= & o_n(1)\nonumber,
	\end{eqnarray}
	where $v_n=u_n-u.$ Hence, $\{v_n\}$ is a 
	$(PS)_{c-I_{\lambda}(u)}$ sequence.
\end{lemma}
\begin{proof}
As the first step, note that
	\begin{align*}
	I_{\lambda}(v_n)-I_{\lambda}(u_n)+I_{\lambda}(u) &=\frac{1}{2}\left( \|v_n\|^2_{\lambda}-\|u_n\|^2_{\lambda}+\|u\|^2_{\lambda}\right)&\\
	&\ \ -\int_{\mathbb{R}^N}\left( F(v_n)-F(u_n)+F(u)\right)dx &\\
	&=o_n(1)-\int_{B_R(0)}\left( F(v_n)-F(u_n)+F(u)\right)dx&\\
	&\ \ -\int_{\mathbb{R}^N \setminus B_R(0)}\left( F(v_n)-F(u_n)+F(u)\right)dx ,&
	\end{align*}
	where $ R> $ 0 will be fixed later on. Once $u_n \rightharpoonup u$ in $E_{\lambda}$, we have
	\begin{center}
		\begin{itemize}
			\item	$u_n \rightarrow u\ \mbox{in}\ L^p(B_R(0)) \ \mbox{for}\ 1 \leq p <2_*;$
			\item $u_n(x) \rightarrow u(x)\ a.e. \ \mbox{in}\ \mathbb{R}^N.$
		\end{itemize}
	\end{center}
	Moreover, there are $h_1 \in L^2(B_R(0))$ and $ h_2 \in L^q(B_R(0))$ such that 
	$$
	\left|u_n(x)\right|\leq h_1(x), h_2(x) \quad \mbox{a.e. in } \quad \mathbb{R}^{N}.
	$$
	By Lebesgue's Theorem,
	\begin{equation}\label{8}
	\int_{B_R(0)} \left| F(v_n)-F(u_n)+F(u)\right|\,dx \rightarrow 0.
	\end{equation}
	
	On the other hand, from $(f_1)-(f_2)$, given $\epsilon >0$, there is $C_\epsilon>0$ satisfying
	$$
	\left|F\left(v_n\right)-F\left(u_n\right)\right| \leq \epsilon \left(\left|u_n\right|+\left|u\right|\right)\left|u\right|+C_{\epsilon}\left(\left|u_n\right|+\left|u\right|\right)^{q-1}\left|u\right|.
	$$
The above estimate combined with the boundedness of $\{u_n\}$ and  Sobolev embeddings gives  
	\begin{align*}
		\int_{\mathbb{R}^N \setminus B_R(0)}\left|F \left(v_n\right)- F\left(u_n\right) \right|dx \leq &\,\,\,\, \epsilon C_1 \left( \|u\|_{L^2(\mathbb{R}^N \setminus
			B_R(0))}+\|u\|_{L^2(\mathbb{R}^N \setminus
			B_R(0))}^2\right)&\\
		& C_\epsilon\left(\|u\|_{L^q(\mathbb{R}^N \setminus
			B_R(0))}+\|u\|_{L^q(\mathbb{R}^N
			\setminus B_R(0))}^q \right).\\
		\end{align*}
The above estimate permits to fix $R>0$ large enough verifying 
	 $$ 
	\int_{\mathbb{R}^N
		\setminus B_R(0)}\left|F \left(v_n\right)-
	F\left(u_n\right) \right|dx \leq \epsilon.
	$$ 
	By $(f_2)$, 
	$$
	\int_{\mathbb{R}^N \setminus B_R(0)}\left|
	F\left(u\right) \right|dx \leq \epsilon
	\left|\left|u\right|\right|_{L^2(\mathbb{R}^N \setminus
		B_R(0))}^2+C_{\epsilon}\|u\|_{L^q(\mathbb{R}^N \setminus
		B_R(0))}^q.
	$$ 
	Then, increasing $R$ if necessary, we can assume that
	$$
	\int_{\mathbb{R}^N \setminus B_R(0)}\left| F\left(u\right)
	\right|dx \leq \epsilon.
	$$ 
	Hence,
	$$ 
	\int_{\mathbb{R}^N
		\setminus B_R(0)}\left|F \left(v_n\right)-
	F\left(u_n\right) +F\left(u\right)\right|dx \leq \epsilon, \quad \forall n \in \mathbb{N}.
	$$
	By arbitrariness of $ \epsilon, $ it follows that
	\begin{equation}\label{9}
	\limsup_{n \rightarrow +\infty} \int_{\mathbb{R}^N
		\setminus B_R(0)}\left|F \left(v_n\right)-
	F\left(u_n\right) +F\left(u\right)\right|dx=0.
	\end{equation}
	From (\ref{8}) and (\ref{9}), we get the first limit. The second one follows by exploring the same type of arguments and the growth of $f'$.
\end{proof}

\begin{lemma}\label{l3}
	Let $\{u_n\}$ be a $(PS)_c$ sequence for $I_{\lambda}$. Then $c=0$, or there exists $c_*>0$ independent of $\lambda,$  such that  $c \geq c_*$ for all $\lambda >0.$
\end{lemma}

\begin{proof}
	By Lemma $\ref{l1}$,  $c \geq 0$. Supposing $c>0$, we get the inequality 
	\begin{align*}
	c+o_n(1)\left|\left|u_n\right|\right|_{\lambda}& \geq I_{\lambda}(u_n)-\frac{1}{\theta}I'_{\lambda}(u_n)u_n
	\geq \left( \frac{\theta-2}{2\theta}\right)\left|\left|u_n\right|\right|^{2}_{\lambda},&
	\end{align*}
which leads to
	\begin{equation}\label{10}
	\limsup_{n \rightarrow + \infty} \left|\left|u_n\right|\right|_{\lambda}^2 \leq \frac{2c\theta}{\theta-2}.
	\end{equation}
	On the other hand, the growth of $f$ together with the Sobolev embedding gives
	$$
	I'_{\lambda}(u)u \geq \frac{1}{2}\left|\left|u \right|\right|_{\lambda}^2-K\left|\left|u \right|\right|_{\lambda}^q,
	$$
	for some positive constant $K$.  Thus, there exists $\delta>0$ such that
	\begin{equation}\label{11}
	I'_{\lambda}(u)u \geq
	\frac{1}{4}\left|\left|u\right|\right|_{\lambda}^2,\ \mbox{for}\
	\left|\left|u\right|\right|_{\lambda} < \delta.
	\end{equation}
	Setting $c_*= \delta^2\frac{\theta -2}{2\theta}$ and 
	$c<c_*$, $(\ref{10})$ yields  
	$$
	\limsup_{n \rightarrow + \infty} \left|\left|u_n\right|\right|_{\lambda}^2 < \delta^2,
	$$
	implying that for $n $ large enough,
	\begin{equation}\label{12}
	\left|\left|u_n\right|\right|_{\lambda}\leq \delta.
	\end{equation}
	Hence, (\ref{11}) and (\ref{12}) combine to give 
	$$
	I'_{\lambda}(u_n)u_n \geq
	\frac{1}{4}\left|\left|u_n\right|\right|_{\lambda}^2,
	$$ 
	leading to	
	$$
	\left|\left|u_n\right|\right|_{\lambda}^2 \rightarrow 0.
	$$
Thus 
	$$
	I_{\lambda}(u_n) \rightarrow I_{\lambda}(0)=0, 
	$$ 
	which contradicts the hypothesis that  $ \{u_n \} $ is a $ (PS) _c $ sequence with
	$c>0$. Therefore, $c \geq c_*.$
\end{proof}

\begin{lemma}\label{l4}
	Let $\{u_n\}$ be a $(PS)_c$ sequence for $I_{\lambda}.$ Then,
	there exists $\delta_0 >0$ independent
	of $\lambda,$ such that $$\liminf_{n
		\rightarrow + \infty}
	\left|\left|u_n\right|\right|_{L^{q}(\mathbb{R}^N)}^q \geq
	\delta_0c.$$
\end{lemma}

\begin{proof}
	By $(f_1)$ and $(f_2),$ given $\epsilon >0$, there is $C_\epsilon>0$ such that
	$$
	\frac{1}{2}f(t)t-F(t) \leq \epsilon\left|t\right|^2+C_{\epsilon}\left|t\right|^{q},\ \forall t \in \mathbb{R}.
	$$
	Then,
	\begin{equation}\label{13}
	c \leq \liminf_{n \rightarrow
		+\infty}\left(\epsilon\left|\left|u_n\right|\right|_{\lambda}^2+C_{\epsilon}\left|\left|u_n\right|\right|_{L^q(\mathbb{R}^N)}^{q}\right).
	\end{equation}
	On the other hand,  by $(f_3)$,
	\begin{equation}\label{14}
	I_{\lambda}(u_n)-\frac{1}{\theta}I'_{\lambda}(u_n)u_n\geq
	\left(\frac{1}{2}-\frac{1}{\theta}\right)\left|\left|u_n\right|\right|_{\lambda}^2.
	\end{equation}
	Combining $(\ref{13})$ with $(\ref{14})$, we get  
	$$
	c \leq \frac{2
		\epsilon c \theta}{\theta-2}+C_{\epsilon}\liminf_{n \rightarrow
		+\infty}\left|\left|u_n\right|\right|_{L^q(\mathbb{R}^N)}^{q}.
	$$
	Thereby, for $\epsilon$ small enough,
	$$\liminf_{n \rightarrow
		+\infty}\left|\left|u_n\right|\right|_{L^q(\mathbb{R}^N)}^{q} \geq
	\frac{c}{C_\epsilon}\left(1-\frac{2\epsilon
		\theta}{\theta-2}\right)>0.
		$$
Now, the lemma follows fixing
	$$
	\delta_0=\frac{1}{C_\epsilon}\left(1-\frac{2\epsilon \theta}{\theta-2}\right).
	$$ 
\mbox{\hspace{13 cm}} \end{proof}
\begin{lemma}\label{l5}
	Let $c_1>0$ be a constant independent of $\lambda$. Given $\epsilon >0$,
	there exist $\Lambda=\Lambda(\epsilon)>0$ and $R=R(\epsilon,c_1)>0$ such that,
	if $\{u_n\}$ is a $(PS)_c$ sequence for $I_{\lambda}$ with $c \in [0,c_1],$  then
	$$
	\limsup_{n \rightarrow +\infty}\left|\left|u_n\right|\right|_{L^q(\mathbb{R}^N \setminus B_R(0))}^{q} \leq \epsilon,\quad \forall \lambda \geq \Lambda.
	$$
\end{lemma}
\begin{proof} \, 	For each $R >0$, fix
	$$
	A(R)=\{x \in \mathbb{R}^N/ \left|x\right|> R\ \mbox{and}\ V(x) \geq M_0\}
	$$
	and
	$$
	B(R)=\{x \in \mathbb{R}^N/ \left|x\right|> R\ \mbox{and}\ V(x) < M_0\}.
	$$
	Then,
	\begin{align}\label{15}
	\int_{A(R)}|u_n|^2dx &\leq \frac{1}{(\lambda M_0 +1)}\int_{\mathbb{R}^N}(\lambda V(x) +1)|u_n|^2dx&\nonumber\\
	&\leq \frac{1}{(\lambda M_0 +1)}\left|\left|u_n\right|\right|_{\lambda}^2&\\
	&\leq \frac{1}{(\lambda M_0 +1)}\left[ \left( \frac{1}{2}-\frac{1}{\theta}\right)^{-1}c+o_n(1)\right]\nonumber&\\
	&\leq \frac{1}{(\lambda M_0 +1)}\left[ \left( \frac{1}{2}-\frac{1}{\theta}\right)^{-1}c_1+o_n(1)\right].\nonumber&
	\end{align}
	As $ c_1 $  is independent of $ \lambda, $  by $(\ref{15})$ there is  $\Lambda>0$ such that 
	\begin{equation}\label{16}
	\limsup_{n \to +\infty}\int_{A(R)}|u_n|^2dx < \frac{\epsilon}{2}, \quad \forall \lambda \geq \Lambda .
	\end{equation}
	On the other hand, using the H\"older inequality for $p\in \left[1,2_*/2\right]$ , we obtain
	\begin{align}\label{17}
	\int_{B(R)}|u_n|^2dx &\leq \left(\int_{B(R)}\left|u_n\right|^{2p}dx\right)^{\frac{1}{p}}\left| B(R)\right|^{\frac{1}{p'}}\nonumber . &
	\end{align}
	Now,  using the continuous embedding $E_{\lambda}\hookrightarrow L^{2p}(\Omega),$ it follows that
	\begin{align*}
	\int_{B(R)}|u_n|^2dx&\leq \beta \left|\left|u_n\right|\right|_{\lambda}^2\left| B(R)\right|^{\frac{1}{p'}},&
	\end{align*}
	where $\beta$ is a positive constant. From (\ref{14}), 
	\begin{align*}
	\int_{B(R)}|u_n|^2dx
	&\leq \beta c_1\left( \frac{1}{2}-\frac{1}{\theta}\right)^{-1}\left|
	B(R)\right|^{\frac{1}{p'}}+o_n(1).\nonumber&
	\end{align*}
Now, by $(V_3)$, we know that 
$$
\left|B(R)\right|\rightarrow 0 \quad \mbox{when} \quad R \rightarrow +\infty.
$$ 	
Therefore, we can choose $R$ large enough, such that
	\begin{equation}\label{17}
	\limsup_{n \to +\infty}\int_{B(R)}|u_n|^2dx < \frac{\epsilon}{2}.
	\end{equation}
Gathering (\ref{16}) and (\ref{17}), we find
$$
\limsup_{n \to +\infty}\int_{\mathbb{R}^N \setminus B_R(0)}|u_n|^2dx < \epsilon.
$$
The last inequality combined with interpolation leads to 
$$
\limsup_{n \to +\infty}\int_{\mathbb{R}^N \setminus B_R(0)}\left|u_n\right|^q dx < \epsilon, \lambda \geq \Lambda
$$
increasing $R$ and $\Lambda$ if necessary.  
\end{proof}
\begin{proposition}\label{p1}
	Given $c_1>0,$ there exists $\Lambda=\Lambda(c_1)$ such that $I_{\lambda}$
	verifies the  $(PS)_c$ condition for all $c \in [0,c_1]$ and  $\lambda
	\geq \Lambda.$
\end{proposition}

\begin{proof}
	Let $\{u_n\}$ be a $(PS)_c$ sequence. By Lemma \ref{l1}, $\{u_n\}$ is bounded and consequently, passing to a subsequence if necessary,
	$$\left\{ \begin{array}{c}
	u_n\rightharpoonup u\ \ \mbox{in}\ E_{\lambda};\\
	u_n(x)\rightarrow u(x)  a.e. \ \mbox{in}\ \mathbb{R}^N;\\
	u_n\rightarrow u \ \mbox{in}\ L^{s}_{loc}(\mathbb{R}^N) \quad \mbox{for} \quad  1 \leq s <2_*.
	\end{array}
	\right.$$
	Then $I_{\lambda}'(u)=0$  and $I_{\lambda}(u)\geq 0$, because
$$
	I_{\lambda}(u)=I_{\lambda}(u)-\frac{1}{\theta}I'_{\lambda}(u)u\geq \left( \frac{1}{2}-\frac{1}{\theta}\right)\|u\|^{2}_{\lambda} \geq 0.
	$$
	Taking $v_n=u_n-u,$ we have by Lemma \ref{l2}
	that $\{v_n\}$ is a $(PS)_{d}$ sequence, with $d=c-I_{\lambda}(u)$. Furthermore, 
	$$
	0\leq d=c-I_{\lambda}(u)\leq c \leq c_1.
	$$
	We claim that $d=0$. Indeed, otherwise $d>0$. Thereby, by Lemmas \ref{l3} and \ref{l4}, $d \geq c_*$
	and 
	\begin{equation}\label{18}
	\liminf_{n \rightarrow +\infty} \| v_n\|_{L^q(\mathbb{R}^N)}^q \geq
	\delta_0c_*>0.
	\end{equation}
	Applying the Lemma \ref{l5} with $\epsilon=\frac{\delta_0c_*}{2}>0,$
	there exist $\Lambda, R>0$ such that 
	\begin{equation}\label{19}
	\limsup_{n \rightarrow +\infty} \| v_n\|_{L^q(\mathbb{R}^N)\setminus
		B_R(0)}^q \leq \frac{\delta_0c_*}{2}, \quad \mbox{for} \quad \lambda \geq \Lambda.
	\end{equation}
	Combining $(\ref{18})$ with $(\ref{19})$, we obtain
	$$\liminf_{n \rightarrow
		+\infty} \left|\left| v_n\right|\right|_{L^q(B_R(0))}^q \geq
	\frac{\delta_0c_*}{2}>0,
	$$ 
	which is an absurd, because as $v_n \rightharpoonup 0$ in $E_\lambda$, and the compact embedding   
	\,\,\,\,\,\,\,\,\,\,\,\,\,\,\,\,\,\,\,\,\,\,\,\,\,$E_{\lambda}\hookrightarrow L^q(B_R(0))$ ensures that
	$$
	\liminf_{n
		\rightarrow +\infty} \| v_n\|_{L^q(B_R(0))}^q=0.
	$$ 
	Therefore $d=0$  and $\{v_n\}$ is a $(PS)_0$ sequence. Then, by Corollary
	\ref{c1}, $v_n \rightarrow 0$ in $E_\lambda$, or equivalently, $u_n \rightarrow u$ in $E_\lambda$, showing that for $\lambda$ large enough, $I_{\lambda}$
	satisfies the $(PS)_c$ condition for all $c \in [0,c_1].$
\end{proof}

\section{The $(PS)_{\infty}$ Condition}
In this section, we will study  the behavior of a  $(PS)_{\infty}$ sequence, that is, a sequence $\{u_n\} \subset H^2(\mathbb{R}^N)$ satisfying:
\begin{align*}
&u_n \in E_{\lambda_n}\ \mbox{and}\ \lambda_{n} \rightarrow +\infty;&\\
&I_{\lambda_n}(u_n)\rightarrow c,\quad \mbox{for some} \quad c\in [0,c_{\Gamma}];&\\
&\|I_{\lambda_n}'(u_n)\|_{E'_{\lambda_n}} \rightarrow 0,&
\end{align*}
where $c_{\Gamma}$ is a positive constant, which will be defined in the next section and it is independent of $\lambda$.
\begin{proposition}\label{p2}
	Let $\{u_n\}$ be a $(PS)_{\infty}$ sequence for $I_{\lambda}$.
	Then, there is a subsequence of $\{u_n\}$ , still denoted by itself,
	and $u \in H^2(\mathbb{R}^N)$ such that $$u_n \rightharpoonup u\ \mbox{in}\  H^2(\mathbb{R}^N).$$
	Moreover,
	\begin{description}
		\item[i)]$u\equiv 0$ in $\mathbb{R}^N \setminus \Omega_{\Gamma}$ and $u$ is a solution of
		\begin{equation}\label{20}
		\left\{ \begin{array}{c}
		\Delta^2 u  +u  =  f(u),\mbox{in}\ \Omega_j, \ \\
		u=\dfrac{\partial u}{\partial \eta} =0,\ \mbox{on}\  \partial\Omega_j,
		\end{array}
		\right.
		\end{equation}
		for all $j \in \Gamma;$
		\item[ii)]$\left|\left| u_n-u\right|\right|^{2}_{\lambda_{n}} \rightarrow 0.$ 
		\item[iii)]$\left\lbrace u_n\right\rbrace $ also satisfies
		\begin{align*}
		&\lambda_n \int_{\mathbb{R}^N}V(x)\left|u_n\right|^2dx \rightarrow 0,\ n \rightarrow +\infty&\\
		&\left|\left|u_n\right|\right|^2_{\lambda_n,\mathbb{R}^N\setminus \Omega_{\Gamma}}\rightarrow 0 &\\
		&\left|\left|u_n\right|\right|^2_{\lambda_n,\Omega'_j}\rightarrow \int_{\Omega_j}\left[ \left|\Delta u\right|^2+\left|u\right|^2\right]dx,\ \forall j\in \Gamma.&
		\end{align*}
	\end{description}
\end{proposition}

\begin{proof}
	In what follows, we fix $c \in [0,c_\Gamma]$ verifying
	$$
	I_{\lambda_n}(u_n)\rightarrow c\ \mbox{and}\ \|I'_{\lambda_n}(u_n)\|_{E'_{\lambda_n}} \rightarrow 0.
	$$
	Then, there exists $n_0 \in \mathbb{N}$ such that, 
	$$
	I_{\lambda_n}(u_n)-\frac{1}{\theta}I'_{\lambda_n}(u_n)u_n \leq c+1+\left|\left|u_n\right|\right|_{\lambda_n}, \quad \forall n \ge n_0.
	$$
	On the other hand, from $f_3)$, 
	$$
	I_{\lambda_n}(u_n)-\frac{1}{\theta}I'_{\lambda_n}(u_n)u_n \geq \left(\frac{1}{2}-\frac{1}{\theta}\right)\left|\left|u_n\right|\right|^2_{\lambda_n},\ \forall n \in \mathbb{N}.
	$$
	So, for $n\ge n_0$,
	$$
	\left(\frac{1}{2}-\frac{1}{\theta}\right)\left|\left|u_n\right|\right|^2_{\lambda_n} \leq c+1+\left|\left|u_n\right|\right|_{\lambda_n},
	$$
	implying that $\{\left|\left|u_n\right|\right|_{\lambda_n}\}$ is bounded in $\mathbb{R}$. As
	$$
	\left|\left|u_n\right|\right|_{\lambda_n} \geq \left|\left|u_n\right|\right|_{H^2(\mathbb{R}^N)}, \ \forall n \in \mathbb{N},
	$$
	$\{u_n\}$ is also bounded in $H^2(\mathbb{R}^N)$, and so, there exists a subsequence of $\{u_n\}$, still denoted by itself, and $u \in H^{2}(\mathbb{R}^{N})$ such that
	$$
	u_n \rightharpoonup u\ \mbox{in}\ H^2(\mathbb{R}^N).
	$$
	To show $(i)$, we fix for each $m \in \mathbb{N}^*$
	the set
	$$
	C_m=\left\{x\in \mathbb{R}^N/ V(x) > \frac{1}{m}\right\}.
	$$
	Hence
	$$\mathbb{R}^N\setminus \overline{\Omega}=\bigcup_{m=1}^{+\infty}C_m.$$
	Note that,
	\begin{align*}
	\int_{C_m}\left|u_n\right|^2dx&=\int_{C_m}\frac{\lambda_n V(x)+1}{\lambda_n V(x)+1}\left|u_n\right|^2dx&\\
	&\leq \frac{1}{\frac{\lambda_n}{m}+1}\left|\left|u_n\right|\right|_{\lambda_n}^2&\\
	&\leq \frac{mM}{{\lambda_n}+m},&
	\end{align*}
	where $M=\sup_{n \in \mathbb{N}}\|u_n\|_{\lambda_n}^2.$ By Fatou's Lemma
	\begin{align*}
	\int_{C_m}\left|u\right|^2dx&\leq \liminf_{n \rightarrow +\infty}\int_{C_m}\left|u_n\right|^2dx&\\
	&\leq \liminf_{n \rightarrow +\infty}\frac{mM}{{\lambda_n}+m}=0.&
	\end{align*}
	Therefore, $u=0$ almost everywhere in $C_m$, and consequently, $u=0$ almost everywhere in $\mathbb{R}^N\setminus \overline{\Omega}.$
	Besides, fixing $\varphi \in C_{0}^{\infty}(\mathbb{R}^N\setminus \overline{\Omega})$, we have
	$$
	\int_{\mathbb{R}^N\setminus \overline{\Omega}}\nabla u(x)\varphi(x)dx=-\int_{\mathbb{R}^N\setminus \overline{\Omega}} u(x)\nabla\varphi(x)dx=0,
	$$
	from where it follows that
	$$
	\nabla u(x)=0,\ a.e. \ \mbox{in}\ \mathbb{R}^N\setminus \overline{\Omega}.
	$$
	Since $\partial \Omega$ is smooth ,  $u \in H^2(\mathbb{R}^N\setminus \overline{\Omega})$ and $\nabla u \in H^1(\mathbb{R}^N\setminus \overline{\Omega}),$  by Trace Theorem , there are constants $K_1,K_2>0$ satisfying 
	$$
	\left|\left|u\right|\right|_{L^2(\partial \Omega)} \leq K_1\left|\left|u\right|\right|_{H^2(\mathbb{R}^N\setminus \overline{\Omega})}=0,
	$$
	and
	$$
	\left|\left|\nabla u\right|\right|_{L^2(\partial \Omega)} \leq K_2 \left|\left|\nabla u\right|\right|_{H^1(\mathbb{R}^N\setminus \overline{\Omega})}=0,
	$$ 
showing that $u \in H_0^2(\Omega).$ To complete the proof of $i)$, consider a test function $\varphi \in C_{0}^{\infty}(\Omega)$ and note that
	\begin{equation}\label{21}
	I'_{\lambda_n}(u_n)\varphi=\int_{\Omega}\left[\Delta u_n \Delta \phi + u_n \varphi\right]dx-\int_{\Omega}f(u_n)\varphi dx.
	\end{equation} 
Since $\left\lbrace u_n\right\rbrace $ is a $(PS)_{\infty}$ sequence, we derive that
	\begin{equation}\label{22}
	I'_{\lambda_n}(u_n)\varphi \rightarrow 0.
	\end{equation}
Recalling that $u_n \rightharpoonup u$ in $H^2(\mathbb{R}^{N})$, we must have 
	\begin{equation}\label{23}
	\int_{\Omega}\left[\Delta u_n \Delta \varphi + u_n \varphi\right]dx \rightarrow \int_{\Omega}\left[\Delta u \Delta \varphi + u \varphi\right]dx
	\end{equation}
	and
	\begin{equation}\label{24}
	\int_{\Omega}f(u_n)\varphi dx \rightarrow \int_{\Omega}f(u)\varphi dx.
	\end{equation}
	Therefore, from (\ref{21})-(\ref{24}), 
	$$
	\int_{\Omega}\left[\Delta u \Delta \varphi +u \phi\right]dx=\int_{\Omega}f(u)\varphi dx,\ \forall \varphi \in C_0^{\infty}(\Omega).
	$$
	As $C_{0}^{\infty}(\Omega)$ is dense in $H_0^2(\Omega)$, the above equality gives
	$$
	\int_{\Omega}\left[\Delta u \Delta v +uv\right]dx=\int_{\Omega}f(u)v dx,\ \forall v \in H_0^2(\Omega),
	$$
	showing that $u$ is a weak solution of the problem
	\begin{equation}\label{25}
	\left\{ \begin{array}{c}
	\Delta^2 u  +u  =  f(u),\, \mbox{in} \, \Omega_j, \ \\
	u=\dfrac{\partial u}{\partial \eta} =0,\ \mbox{on}\  \partial\Omega_j.
	\end{array}
	\right.
	\end{equation}
	For $ii)$, note that
	\begin{align} \label{26}
	\left|\left|u_n-u\right|\right|_{\lambda_n}^2&= \left|\left|u_n\right|\right|_{\lambda_n}^2+\left|\left|u\right|\right|_{\lambda_n}^2-2\int_{\mathbb{R}^N}\left[\Delta u_n\Delta u+ (\lambda_n V(x)+1)u_nu\right]dx.& 
	\end{align}
From $i)$, 
$$
	\|u\|_{\lambda_n}^2=\|u\|^2_{H_0^2(\Omega)},
$$
and so, 
	$$
	\int_{\mathbb{R}^N}\left[\Delta u_n\Delta u+ (\lambda_n V(x)+1)u_nu\right]dx=\|u\|_{H_0^2(\Omega)}^2+o_n(1).
	$$
	Thus, we can rewrite (\ref{26}) as
	\begin{align} \label{27}
	\left|\left|u_n-u\right|\right|_{\lambda_n}^2&= \left|\left|u_n\right|\right|_{\lambda_n}^2-\left|\left|u\right|\right|_{H_0^2(\Omega)}^2+o_n(1).& 
	\end{align}
Gathering the boundedness of $\{\|u_n\|_{\lambda_n}\}$ with the limit $ \|I'_{\lambda_n}(u_n)\|_{E'_{\lambda_n}}\rightarrow 0,$ we find the limit 
	$$
	I'_{\lambda_n}(u_n)u_n \rightarrow 0.
	$$
	Hence,
	\begin{equation} \label{28}
	\left|\left|u_n\right|\right|_{\lambda_n}^2= I'_{\lambda_n}(u_n)u_n+\int_{\mathbb{R}^N}f(u_n)u_ndx =\int_{\mathbb{R}^N}f(u_n)u_ndx+o_n(1).
	\end{equation}
	On the other hand,  we know that the limit  $	I'_{\lambda_n}(u_n)u \to 0 $ is equivalent to 
  $$
	\int_{\Omega}\left[ \Delta u_n \Delta u + u_nu\right]dx-\int_{\Omega}f(u_n)udx=o_n(1),
	$$
	which leads to 
	\begin{equation} \label{29}
	\int_{\mathbb{R}^N}\left[ \left|\Delta u\right|^2 +
	\left|u\right|^2\right]dx=\int_{\mathbb{R}^N}f(u)udx.
	\end{equation}
	Combining (\ref{27}) with (\ref{28}) and (\ref{29}), we see that 
	$$
	\left|\left|u_n-u\right|\right|_{\lambda_n}^2=\int_{\mathbb{R}^N}f(u_n)u_ndx-\int_{\mathbb{R}^N}f(u)udx+o_n(1).
	$$
	The same arguments used in the proof of Lemma \ref{l5} gives   
	$$
	\int_{\mathbb{R}^N}f(u_n)u_ndx\rightarrow \int_{\mathbb{R}^N}f(u)udx,
	$$
	finishing the proof of $ii)$. 
	Finally, to prove $iii),$ it is enough to use the inequality below 
	$$
	\lambda_n\int_{\mathbb{R}^N}V(x)\left|u_n\right|^2dx=\lambda_n\int_{\mathbb{R}^N}V(x)\left|u_n-u\right|^2dx
	\leq \|u_n-u\|_{\lambda_n}^2 \rightarrow 0.
	$$ \hspace{13.5cm}
\end{proof}

\section{A special minimax level}
In this section, we denote by $I_j:H_0^2(\Omega_j) \rightarrow
\mathbb{R} $ and $I_{\lambda,j}:H^2(\Omega'_j)  \rightarrow \mathbb{R}
$ the functionals given by
$$
I_j(u)=\frac{1}{2}\int_{\Omega_j}\left[\left|\Delta u\right|^2+\left| u\right|^2\right]dx-\int_{\Omega_j}F(u)dx
$$
and
$$
I_{\lambda,j}(u)=\frac{1}{2}\int_{\Omega'_j}\left[\left|\Delta u\right|^2+(\lambda V(x)+1)\left| u\right|^2\right]dx-\int_{\Omega'_j}F(u)dx.
$$

It is easy to show that $I_j$ and  $I_{\lambda,j}$ satisfy the mountain pass geometry. Hereafter, we denote by $c_j$ and
$c_{\lambda,j}$ the mountain pass levels related to the functionals  $I_j$ and $I_{\lambda,j}$ respectively.

Since $I_j$ and $I_{\lambda,j}$ satisfy  the
Palais-Smale condition, from Mountain Pass Theorem due to
Ambrosetti-Rabinowitz, there exist $w_j \in H_0^2(\Omega_j)$ and  $v_j
\in H^2(\Omega'_j)$ satisfying 
$$
I_j(w_j)=c_j,\ I_{\lambda,j}(v_j)=c_{\lambda,j}\ \mbox{and}\ I'_j(w_j)=I'_{\lambda,j}(v_j)=0.
$$

In what follows, $c_{\Gamma}=\sum_{j=1}^{l}c_j$ and $R>0$ is a constant large enough verifying  
$$0
< I_j(\frac{1}{R}w_j), I_j(Rw_j)< c_j, \forall j \in \Gamma.
$$
Hence, by definition of $c_j$, 
$$
\max_{s \in
	[1/R^2,1]}I_j(sRw_j)=c_j,\ \forall j \in \Gamma.
$$ 
Consider $\Gamma=\{1,2,\cdots,l\}$, with $l \leq k$ and fix 
$$
\gamma_0(s_1,s_2,\cdots,s_l)(x)=\sum_{j=1}^{l}s_jRw_j(x),\ \forall (s_1,\cdots,s_l)\in [1/R^2,1]^l.
$$
From now on, we denote by $\Gamma_*$ the class of continuous path $\gamma \in C([1/R^2,1],E_{\lambda}\setminus \{0\})$ satisfying the following conditions: 
$$
\gamma=\gamma_0\ \mbox{on}\ \partial([1/R^2,1]^l) \leqno{(a)}
$$
and 
$$
I_{\lambda,\mathbb{R}^N\setminus \Omega'_{\Gamma}}(\gamma(s_1,\cdots,s_l))\ge 0, \leqno{(b)}
$$
where $I_{\lambda,\mathbb{R}^N\setminus \Omega'_{\Gamma}}: H^{2}(\mathbb{R}^N\setminus \Omega'_{\Gamma}) \to \mathbb{R}$ is the functional defined  by
$$
I_{\lambda,\mathbb{R}^N\setminus \Omega'_{\Gamma}}(u)=\frac{1}{2}\int_{\mathbb{R}^N\setminus \Omega'_{\Gamma}}\left[\left|\Delta u\right|^2+(\lambda V(x)+1)\left| u\right|^2\right]dx-\int_{\mathbb{R}^N\setminus \Omega'_{\Gamma}}F(u)dx.
$$

Using the class $\Gamma_*$, we define the following minimax level
$$
b_{\lambda,\Gamma}=\inf_{\gamma \in \Gamma_*}\max_{(s_1,\cdots,s_l) \in [1/R^2,1]^l}I_{\lambda}(\gamma(s_1,\cdots,s_l)).
$$

Notice that $\Gamma_* \neq \emptyset$, because $\gamma_0 \in \Gamma_*$. 
\begin{lemma}\label{l6}
	For each $\gamma \in \Gamma_*,$ there is
	$(t_1,\cdots,t_l)\in[1/R^2,1]^l$ verifying 
	$$
	I'_{\lambda,j}(\gamma(t_1,\cdots,t_l))\gamma(t_1,\cdots,t_l)=0,\ \mbox{for}\ j\in\{1,\cdots,l\}.
	$$
\end{lemma}
\begin{proof}
	Given $\gamma \in \Gamma_*$, consider the map $\widetilde{\gamma}:[1/R^2,1]^l \rightarrow \mathbb{R}^l$ defined by
	$$
	\widetilde{\gamma}(s_1,\cdots,s_l)=\left(I'_{\lambda,1}(\gamma(s_1,\cdots,s_l))\gamma(s_1,\cdots,s_l), \cdots, I'_{\lambda,l}(\gamma(s_1,\cdots,s_l))\gamma(s_1,\cdots,s_l)	\right).
	$$ 
	For $(s_1,\cdots,s_l) \in \partial([1/R^2,1]^l),$ we know that
	$$
	\gamma(s_1,\cdots,s_l)=\gamma_0(s_1,\cdots,s_l).
	$$
	Then,
	$$
	I'_{\lambda,j}(\gamma_0(s_1,\cdots,s_l))(\gamma_0(s_1,\cdots,s_l))=0 \Rightarrow s_j \not \in \{1/R^2,1\}, \forall j \in \Gamma,
	$$
	otherwise,
	$$
	I'_{\lambda,j}(\gamma_0(s_1,\cdots,s_l))(\gamma_0(s_1,\cdots,s_l))=0
	$$
	for $s_{j}=\frac{1}{R^2}$ or $s_{j}=1$, that is, 
	$$
	I'_{j}(\frac{1}{R}w_j)(\frac{1}{R}w_j)=0 \quad \mbox{or} \quad I'_{j}(Rw_j)(Rw_j)=0
	$$
	implying  that 
	$$
	I_{j}(\frac{1}{R}w_j)\geq c_j \quad \mbox{or} \quad I_{j}(Rw_j)\geq c_j,
	$$
	which contradicts the choice of $R$. Hence, 
	$$
	(0,0,\cdots,0) \not 	\in\widetilde{\gamma}(\partial([1/R^2,1]^l)).
	$$ 
	Then, by Topological Degree 
	$$
	deg(\widetilde{\gamma},(1/R^2,1)^l,(0,0,\cdots,0))=(-1)^l\not= 0,
	$$
	from where it follows that there exists $(t_1,t_2,\cdots,t_l)\in (1/R^2,1)^l$
	satisfying
$$
	I'_{\lambda,j}(\gamma(t_1,t_2,\cdots,t_l))(\gamma(t_1,t_2,\cdots,t_l))=0,\ \mbox{for}\ j \in \{1,2,\cdots,l\}.
$$
\mbox{\hspace{13,4 cm}}\end{proof}

\begin{proposition}\label{p3} \mbox{}\\
\noindent $a) \,\, \sum_{j=1}^{l}c_{\lambda,j}\leq b_{\lambda, \Gamma} \leq c_{\Gamma},\, \forall \lambda \geq 1.$ \\
\noindent $b)$ \,\, For $\gamma \in \Gamma_*\ \mbox{and}\ (s_1,\cdots,s_l)\in \partial([1/R^2,1]^l)$, we have 
$$
I_{\lambda}(\gamma(s_1, \cdots,s_l))< c_{\Gamma},\, \forall \lambda \geq 1. 
$$

\end{proposition}

\begin{proof}
	\begin{description}
		\item[a)]Since $\gamma_0 \in \Gamma_*$, 
		\begin{align*}
		b_{\lambda,\Gamma} &\leq \max_{(s_1,\cdots,s_l)\in [1/R^{2},1]^l}I_{\lambda,j}(\gamma_0(s_1,\cdots,s_l))&\\
		&\leq \max_{(s_1,\cdots,s_l)\in [1/R^{2},1]^l}I_{\lambda,j}(\sum_{i=1}^{l}s_iRw_i(x))&\\
		&\leq \sum_{j=1}^{l}\max_{s_j \in [1/R^{2},1]}I_{j}(s_jRw_j(x))&\\
		&\leq \sum_{j=1}^{l}c_j=c_{\Gamma.}&
		\end{align*}
		For each $\gamma \in \Gamma_*$ and  $(t_1,\cdots,t_l)\in [1/R^2,1]^{l}$ as in Lemma \ref{l6}, we find 
		$$
		I_{\lambda,j}(\gamma(t_1,\cdots, t_l)) \geq c_{\lambda,j}, \forall j \in \Gamma.
		$$
		In the last inequality we have used the following the equality below 	
		$$
		c_{\lambda,j}=\inf\{I_{\lambda,j}(u);\ u \in E_{\lambda}\setminus \{0\};\,  I'_{\lambda,j}(u)u=0\}. 
		$$ 
		On the other hand, recalling that $I_{\lambda,\mathbb{R}^N\setminus \Omega'_{\Gamma}}(\gamma(s_1,\cdots,s_l))\geq 0$, we derive that
		$$
		I_{\lambda}(\gamma(s_1, \cdots, s_l)) \geq
		\sum_{j=1}^{l}I_{\lambda,j}(\gamma(s_1, \cdots, s_l)),
		$$ 
		and so,
		$$
		\max_{(s_1,\cdots, s_l)\in [1/R^2,1]^l}I_{\lambda}(\gamma(s_1, \cdots, s_l)) \geq I_{\lambda}(\gamma(t_1, \cdots, t_l)) \geq \sum^{l}_{j=1}c_{\lambda,j}.
		$$ 
		The last inequality combined with the definition of $b_{\lambda, \Gamma}$ gives  
		$$
		b_{\lambda, \Gamma}\geq \sum_{j=1}^{l}c_{\lambda, j},
		$$ 
		This completes the proof of $a).$
		\item[b)] As $\gamma(s_1,\cdots,s_l)=\gamma_0(s_1,\cdots,s_l)$ on $\partial([1/R^2,1]^l),$ we derive that
		$$
		I_{\lambda}(\gamma_0(s_1,\cdots,s_l))=\sum_{j=1}^{l}I_j(s_jRw_j), \forall (s_1,\cdots,s_l) \in \partial([1/R^2,1]^l).
		$$
		Since 
		$$
		I_j(s_jRw_j) \leq c_j, \quad \forall j \in \Gamma
		$$ 
		and there is $j_0 \in \Gamma$, such that $s_{j_0} \in \{1/R^2,1 \}$, we have  
		$$
		I_{\lambda}(\gamma_0(s_1,\cdots,s_l)) < c_{\Gamma}.
		$$
		
	\end{description}
\mbox{\hspace{13 cm}} \end{proof}

\begin{corollary}
	
	$b_{\lambda, \Gamma} \rightarrow c_{\Gamma},$ when $\lambda
	\rightarrow +\infty.$
	
\end{corollary}

\begin{proof}
Using the same arguments found in \cite{D&T}, it is possible to prove that $c_{\lambda, j} \rightarrow c_j$ for each $j \in \Gamma$. Therefore, by Proposition $\ref{p3}$, $b_{\lambda,\Gamma} \rightarrow
	c_{\Gamma}$ when $\lambda \rightarrow +\infty.$
\end{proof}

\section{Proof of the Main Theorem}
Hereafter, we consider
$$
M=1+\sum_{j=1}^{l}\sqrt{\left(\frac{1}{2}-\frac{1}{\theta}\right)c_j},
$$
$$
\overline{B}_{M+1}(0)=\{ u \in E_{\lambda};\|u\|_{\lambda}\leq M+1\},
$$ 
and for small $\mu>0$
$$
A_{\mu}^{\lambda}=\left\{ u \in \overline{B}_{M+1}; \left|\left|u\right|\right|_{\lambda,\mathbb{R}^N\setminus \Omega'_{\Gamma}}\leq \mu,\ I_{\lambda,\mathbb{R}^N\setminus 	\Omega'_{\Gamma}}(u)\geq 0\ \mbox{and}\ \left|I_{\lambda,j}(u)-c_j\right|\leq \mu, \forall j \in  \Gamma \right\},
$$
and
$$
I_{\lambda}^{c_{\Gamma}}=\left\{ u \in E_{\lambda}/ I_{\lambda}(u)\leq c_{\Gamma}\right\}.
$$
Note that $A_{\mu}^{\lambda}\cap I_{\lambda}^{c_{\Gamma}} \neq \emptyset,$  because $w= \sum_{j=1}^{l}w_j \in A_{\mu}^{\lambda}\cap I_{\lambda}^{c_{\Gamma}} .$ Fixing
\begin{equation}\label{30}
0< \mu < \frac{1}{4}\min\{c_j;j \in  \Gamma\},
\end{equation}
we have the following uniform estimate from below for $\|I'_{\lambda}(u)\|$ in the set $\left(A_{2\mu}^{\lambda}\setminus
A_{\mu}^{\lambda}\right)\cap I_{\lambda}^{c_{\Gamma}}.$

\begin{proposition}\label{p4}
	Let $\mu >0$ satisfy $(\ref{30})$. Then, there exist $\sigma_0>0$ independent of $\lambda$ and
	$\Lambda_* \geq 1$ such that
	\begin{equation}\label{31}
	\ \|I'_{\lambda}(u)\| \geq \sigma_0\ \mbox{for}\ \lambda \geq
	\Lambda_*\ \mbox{and for all}\ u \in
	\left(A_{2\mu}^{\lambda}\setminus A_{\mu}^{\lambda}\right)\cap
	I_{\lambda}^{c_{\Gamma}}.
	\end{equation}
\end{proposition}

\begin{proof}
	Arguing by contradiction, suppose that there are $\lambda_n \rightarrow +\infty$  and $u_n \in E_{\lambda_n}$, with
	$$
	u_n \in \left(A_{2\mu}^{\lambda_n}\setminus A_{\mu}^{\lambda_n}\right)\cap I_{\lambda}^{c_{\Gamma}} \quad \mbox{and} \quad \|I'_{\lambda_n}(u_n)\| \rightarrow 0.
	$$
	Since $u_n \in A_{2\mu}^{\lambda_n}$, the sequence
	$\left\{\|u_n\|_{\lambda_n}\right\}$ is bounded. Consequently
	$\left\{I_{\lambda_n}(u_n)\right\}$ is also bounded. Then,  passing to a subsequence if necessary,
	$$
	I_{\lambda_n}(u_n) \rightarrow c \in (-\infty,c_{\Gamma}].
	$$
		By Proposition \ref{p2}, passing to a subsequence if necessary, $u_n
	\rightarrow u$ in $H^2(\mathbb{R}^N)$  and $u \in
	H_{0}^{2}(\Omega_{\Gamma})$ is a solution of the problem (\ref{20}). Moreover,
	\begin{align}
	&\lambda_n \int_{\mathbb{R}^N}V(x)\left|u_n\right|^2dx \rightarrow 0,& \label{32}\\
	&\left|\left|u_n\right|\right|^2_{\lambda_n,\mathbb{R}^N\setminus \Omega_{\Gamma}}\rightarrow 0 & \label{33}\\
	&\left|\left|u_n\right|\right|^2_{\lambda_n,\Omega'_j}\rightarrow \int_{\Omega_j}\left[ \left|\Delta u\right|^2+\left|u\right|^2\right]dx,\ \forall j\in \Gamma.& \label{34}
	\end{align}
	Since $c_{\Gamma}=\sum_{j=1}^{l}c_j$ and $c_j$ is the least energy level for $I_j$, one of the following cases occurs:
	\begin{description}
		\item[i)] $u\left|_{\Omega_j}\neq 0\right.,\ \forall j \in \Gamma,$ or
		\item[ii)]$u\left|_{\Omega_{j_0}}=0\right.,$ for some $j_0 \in \Gamma.$
	\end{description}
	If $i)$ happens, from $(\ref{32})-(\ref{34})$ 
	$$
	I_j(u)=c_j, \quad \forall j \in \Gamma.
	$$
	Hence $u_n \in A_{\mu}^{\lambda_n}$ for $n$ large enough, which is a contradiction. 
	
	If $ii)$ happens, from $(\ref{32})\ \mbox{and}\ (\ref{33})$  $$\left|I_{\lambda_{n},j_0}(u_n)-c_{j_0})\right| \rightarrow c_{j_0} \geq 4\mu,$$
	which contradicts the  hypothesis  $u_n \in A_{2 \mu}^{\lambda_n}\setminus A_{\mu}^{\lambda_n}$ for all $n \in \mathbb{N}$. Since $i)$ or $ii)$ cannot happen, we get an absurd, finishing the proof.
\end{proof}

\begin{proposition} \label{p5}
	Let $\mu$ satisfy $(\ref{30})$ and $\Lambda_* \geq 1$
	the constant given in the Proposition \ref{p3}. Then for $\lambda \geq
	\Lambda_*$, there exists $u_{\lambda}$ a solution of $(\ref{1})$
	satisfying $u_{\lambda}\in A_{\mu}^{\lambda}\cap
	I_{\lambda}^{c_{\Gamma}}.$
\end{proposition}
\begin{proof}
	We will suppose, by contradiction, that there are no critical points of
	$I_{\lambda}$ in $A_{\mu}^{\lambda}\cap I_{\lambda}^{c_{\Gamma}}.$
	By Proposition $\ref{p1}$, $I_{\lambda}$ satisfies the $(PS)_d$ condition
	for $d \in [0,c_{\Gamma}] $ and $\lambda$ large enough. Thereby, there exists $d_{\lambda} >0$ such that
	$$
	\left|\left|I'_{\lambda}(u)\right|\right| \geq d_{\lambda},\
	\forall u \in A_{\mu}^{\lambda}\cap I_{\lambda}^{c_{\Gamma}}.
	$$
	On the other hand, by Proposition \ref{p4}, 
	$$
	\|I'_{\lambda}(u)\|
	\geq \sigma_0,\ \forall u \in (A_{2\mu}^{\lambda} \setminus
	A_{\mu}^{\lambda})\cap I_{\lambda}^{c_{\Gamma}},
	$$ 
	where $\sigma_0$ is independent of $\lambda.$ Now, we define the continuous functions
	$\Psi:E_{\lambda}\rightarrow \mathbb{R}$ and
	$H:I_{\lambda}^{c_{\Gamma}}\rightarrow \mathbb{R}$ by
	\begin{align*}
	\Psi(u)= 1,&  \quad u \in  A_{3\mu/2}^{\lambda},\\
	\Psi(u)=0, & \quad u \not \in A_{2\mu}^{\lambda}, \\
	0 \leq \Psi(u)\leq 1,&\quad \mbox{for}\ u \in E_{\lambda},
	\end{align*}
	and
	$$
	H(u)=\left\{ \begin{array}{cc}
	-\Psi(u)\left|\left|Y(u)\right|\right|^{-1}Y(u),&\ u \in A_{2\mu}^{\lambda},\\
	0,& u\not \in A_{2\mu}^{\lambda},
	\end{array}
	\right.
	$$ 
	where $Y$ is a pseudogradient  vector field for $I_{\lambda}$ on 
	$$
	X=\{ u \in E_{\lambda}; I_{\lambda}(u) \neq 0\}.
	$$ 
	Notice that 
	$$
	\left|\left|H(u)\right|\right| \leq 1 \ \mbox{for all}\ \lambda \geq \Lambda_*\ \mbox{and}\ u \in I_{\lambda}^{c_{\Gamma}}. 
	$$
	The above information ensures the existence of a flow $\eta: [0,+\infty) \times I_{\lambda}^{c_{\Gamma}} \rightarrow  I_{\lambda}^{c_{\Gamma}}$ defined by
	$$
	\left\{ \begin{array}{ccc}
	\dfrac{d\eta(t,u)}{dt}&=& H(\eta(t,u))\\
	\eta(0,u)&=&u\in I_{\lambda}^{c_{\Gamma}},
	\end{array}
	\right.
	$$
	verifying
	\begin{equation}\label{35}
	\dfrac{d I_{\lambda}(\eta(t,u))}{dt}\leq -\Psi(\eta(t,u))\left|\left|I'_{\lambda}(\eta(t,u))\right|\right|\leq 0,
	\end{equation}
	\begin{equation}\label{36}
	\left|\left|\dfrac{d\eta}{dt}\right|\right|= \left|\left|H(\eta)\right|\right|\leq 1,
	\end{equation}
	and
	\begin{equation}\label{37}
	\eta(t,u)=u,\ \forall\ t\geq 0\ \mbox{and}\ u \in I_{\lambda}^{c_{\Gamma}}\setminus A_{2 \mu}^{\lambda}.
	\end{equation}
In what follows, we set
$$
\beta(s_1,\cdots,s_l)=\eta(T,\gamma_0(s_1,\cdots,s_l)),\
	\forall (s_1,\cdots,s_l)\in [1/R^2,1]^l,
$$ 
where $T>0$ will be fixed later on. 

Once
$$
	\gamma_0(s_1,\cdots,s_l)\not \in A_{2 \mu}^{\lambda},\ \forall
	(s_1,\cdots,s_l)\in \partial ([1/R^2,1]^l),
$$ 
we deduce that
$$
	\beta(s_1,\cdots,s_l)=\gamma_0(s_1,\cdots,s_l),\
	\forall (s_1,\cdots,s_l)\in \partial([1/R^2,1]^l).
$$ 
Moreover, 	it is easy to check that
$$
I_{\lambda,\mathbb{R}^N\setminus 	\Omega'_{\Gamma}}(\beta(s_1,\cdots,s_l))\geq 0, \quad \forall (s_1, \cdots,s_l)\in [1/R^2,1]^l,
$$
showing that $\beta \in \Gamma_*$.
	
	Note that $supp(\gamma_0(s_1,\cdots,s_l)) \subset
	\overline{\Omega}_{\Gamma}$ for all $(s_1,\cdots,s_l) \in
	[1/R^2,1]^l$ and  $I_{\lambda}(\gamma_0(s_1,\cdots,s_l))$
	independent of $\lambda \geq \Lambda.$ Furthermore,
	$$
	I_{\lambda}(\gamma_0(s_1, \cdots,s_l))\leq c_{\Gamma}, \forall (s_1, \cdots,s_l)\in [1/R^2,1]^l
	$$
	and
	$$
	I_{\lambda}(\gamma_0(s_1, \cdots,s_l))= c_{\Gamma},\ \mbox{if}\ s_j=1/R, \forall j \in \Gamma. 
	$$
	Therefore,
	$$
	m_0=max\left\{ I_{\lambda}(u);u \in \gamma_0([1/R^2,1]^l)\setminus A_{\mu}^{\lambda}\right\}< c_{\Gamma},
	$$
	and $m_0$ is independent of $\lambda$.
	
	Since there is $K_*$ such that
	$$
	\left|I_{\lambda}(u)-I_{\lambda}(v)\right|\leq K_*\|u-v\|_{\lambda,\Omega'_j},\ \forall u,v \in \overline{B}_{M+1}\ \mbox{and}\ \forall j \in \Gamma,
	$$ 
	we claim that if $T$ is large enough, the estimate below holds 
	\begin{equation} \label{38}
	\max_{(s_1,\cdots,s_l \in
		[1/R^2,1]^l)}I_{\lambda}\left(
	\beta(s_1,\cdots,s_l)\right)\leq
	\max\{m_0,c_{\Gamma}-\frac{1}{2K_*}\sigma_0\mu\}.
	\end{equation}
	Indeed, fix $u=\gamma_0(s_1,\cdots,s_l) \in E_{\lambda}.$ If $u
	\not \in A_{\mu}^{\lambda},$ 
	$$
	I_{\lambda}\left(
	\eta(t,u)\right)) \leq I_{\lambda}\left(
	\eta(0,u)\right))=I_{\lambda}(u)\leq m_0, \quad \forall t \geq 0.
	$$ 
	On the other hand, if $u \in A_{\mu}^{\lambda}$, by setting $\tilde{\eta}(t)=\eta(t,u),$ 
	$\tilde{d}_{\lambda}=\min\{d_{\lambda},\sigma_0\}$ and  
	$T=\frac{\sigma_0 \mu}{2 K_*d_{\lambda}}>0$,  we analyze the following cases: \\
	
	\noindent \textbf{Case 1:} $\tilde{\eta}(t) \in A_{3\mu/2}^{\lambda}, \forall
	t \in [0,T].$ \\
	
	\noindent \textbf{Case 2:} $\tilde{\eta}(t_0) \in \partial
	A_{3\mu/2}^{\lambda},\ \mbox{for some}\ t_0 \in [0,T].$\\

	\noindent {\bf Analysis of the Case 1:} \,  In this case,
	$$
	\Psi(\tilde{\eta}(t))\equiv 1, \quad \forall t \in [0,T]
	$$ 
	and 
	$$
	\| I'_{\lambda}(\tilde{\eta}(t))\| \geq \tilde{d}_{\lambda}, \forall t \in [0,T].
	$$
	Hence,
	$$
	I_{\lambda}(\tilde{\eta}(T))=I_{\lambda}(u)+\int_{0}^{T}\frac{d}{ds}I_{\lambda}(\tilde{\eta}(s))ds \leq c_{\Gamma}-\int_{0}^{T}\tilde{d}_{\lambda}ds,
	$$
	it follows that
	$$
	I_{\lambda}(\tilde{\eta}(T))\leq c_{\Gamma}-\tilde{d}_{\lambda}T=c_{\Gamma}-\frac{1}{2K_*}\sigma_0\mu.
	$$
	
	\noindent {\bf Analysis of the Case  2:} \, Let $0 \leq t_1 \leq t_2 \leq T$ satisfy 
	$\tilde{\eta}(t_1)\in \partial A_{\mu}^{\lambda},$
	$\tilde{\eta}(t_2)\in \partial A_{3\mu/2}^{\lambda}$ and
	$\tilde{\eta}(t)\in  A_{3\mu/2}^{\lambda}\setminus
	A_{\mu}^{\lambda}, \forall t\in [t_1,t_2].$ Then
	\begin{equation}\label{39}
	\|\tilde{\eta}(t_1)-\tilde{\eta}(t_2)\|\geq \frac{1}{2K_*}\mu.
	\end{equation}
	Indeed, denoting $w_1=\tilde{\eta}(t)$ and $w_2=\tilde{\eta}(t_2),$
	it follows that 
	$$
	\|w_2\|_{\lambda,\mathbb{R^N}\setminus
		\Omega'_{\Gamma}}=\frac{3}{2} \mu \quad \mbox{or} \quad 
	\left|I_{\lambda,j_0}(w_2)-c_{j_0}\right|=\frac{3}{2} \mu.
	$$
	From definition of $A_{\mu}^{\lambda},$ we have $\|w_2\|_{\lambda,\mathbb{R^N}\setminus\Omega'_{\Gamma}}\leq \mu. $ Thus,
	$$
	\|w_2-w_1\|_{\lambda} \geq \frac{1}{K_*}\left|I_{\lambda,j_0}(w_2)-I_{\lambda,j_0}(w_1)\right|\geq \frac{1}{2K_*}\mu.
	$$
	By Mean Value Theorem
	\begin{equation}\label{40}
	\|\tilde{\eta}(t_1)-\tilde{\eta}(t_2)\|_{\lambda}\le\left|\left|\dfrac{d\eta}{dt}\right|\right|\left| t_1-t_2\right|.
	\end{equation}
	As $\left|\left|\dfrac{d\eta}{dt}\right|\right| \le 1$, from (\ref{39}) and (\ref{40}),   
	$$
	\left| t_1-t_2\right| \ge \frac{1}{2K_*}\mu.
	$$
	Hence
	$$
	I_{\lambda}(\tilde{\eta}(T))\leq I_{\lambda}(u)-\int_{0}^{T}\Psi(\tilde{\eta}(s))\|I'_{\lambda}(\tilde{\eta}(s))\|ds,
	$$
	and so,
	$$
	I_{\lambda}(\tilde{\eta}(T))\leq c_{\Gamma}-\int_{t_1}^{t_2}\sigma_0ds \leq c_{\Gamma}-\frac{\sigma_0}{2K_*}\mu,
	$$
	proving (\ref{38}).
	
	Thereby,  
	$$
	b_{\lambda,\Gamma}\leq \max_{[1/R^2,1]^l}I_{\lambda}(\widehat{\eta}(s_1,\cdots,s_l))\leq \max\{m_0,c_{\Gamma}-\frac{1}{2K*}\sigma_0 \mu\} < c_{\Gamma},
	$$ 
	which is an absurd, because $b_{\lambda, \Gamma}\rightarrow c_{\Gamma},$ when $\lambda \rightarrow \infty.$
\end{proof}

\vspace{0.5 cm}

Thus, we can conclude that $I_{\lambda}$ has a critical point $u_{\lambda}$ in $A_{\mu}^{\lambda}$ for $\lambda$ large enough.

\vspace{0.5 cm}

\noindent \textbf{Completion of the Proof of Theorem \ref{T1}:}

\vspace{0.5 cm}

From the last proposition there exists $\{u_{\lambda_n}\}$ with
$\lambda_n \rightarrow +\infty$ satisfying:
$$I'_{\lambda_n}(u_{\lambda_n})=0,$$
$$
\|u_{\lambda_n}\|_{\lambda_n,\mathbb{R}^N\setminus \Omega'_{\Gamma}} \rightarrow 0
$$
and
$$
I_{\lambda_n,j}(u_{\lambda_n})\rightarrow c_j, \forall j \in \Gamma.
$$
Therefore, from Proposition \ref{p2},
$$
u_{\lambda_n} \rightarrow u \quad \mbox{in} \quad  H^2(\mathbb{R}^N)\ \mbox{with}\ u\ \in H_{0}^{2}(\Omega_{\Gamma}).
$$
Moreover, $u$ is a nontrivial solution of
\begin{equation} \label{PF}
		\left\{ \begin{array}{c}
		\Delta^2 u  +u  =  f(u),\mbox{in}\ \Omega_j \ \\
		u=\dfrac{\partial u}{\partial \eta} =0,\ \mbox{on}\  \partial\Omega_j,
		\end{array}
		\right.
		\end{equation}
with $I_j(u)=c_j$ for all $i \in \Gamma$.  Now, we claim that $u=0$ in $\Omega_j$, for all $j \notin \Gamma$. Indeed, it is possible to prove that there is $\sigma_1>0$, which is independent of $j$, such that if $v$ is a  nontrivial solution of (\ref{PF}), then
$$
\|v\|_{H_0^{2}(\Omega_j)} \geq \sigma_1.
$$
However, the solution $u$ verifies 
$$
\|u\|_{H^{2}(\mathbb{R}^{N} \setminus \Omega'_\Gamma)}=0,
$$
showing that $u=0$ in $\Omega_j$, for all $j \notin \Gamma$. This finishes the proof of Theorem \ref{T1}.

\end{document}